\documentclass[10pt]{amsart}
\usepackage[utf8]{inputenc}
\usepackage[margin=0.65in]{geometry}
\usepackage{amsmath,amssymb}
\usepackage{xcolor} 
\usepackage{amsthm}
\usepackage{amsfonts}
\usepackage{graphicx}
\usepackage[dvpis]{epsfig}
\usepackage{placeins}
\usepackage[overload]{empheq}
\usepackage{hyperref}

\makeatletter
\newcommand*{\rom}[1]{\expandafter\@slowromancap\romannumeral #1@}
\makeatother

\newcommand{\R}{\mathbb{R}}

\newtheorem{theorem}{Theorem}[section]

\newtheorem{lemma}[theorem]{Lemma}

\theoremstyle{definition}
\newtheorem{definition}[theorem]{Definition}
\newtheorem{example}[theorem]{Example}
\theoremstyle{remark}
\newtheorem{remark}[theorem]{Remark}

\numberwithin{equation}{section}

\DeclareMathOperator*{\argmin}{arg\,min}

\newcommand{\e}{\varepsilon}
\newcommand{\lt}{\left}
\newcommand{\rt}{\right}
\newcommand{\bq}{\begin{equation}}
\newcommand{\eq}{\end{equation}}

\newcommand{\bfx}{{\bf x}}
\newcommand{\bfy}{{\bf y}}

\newcommand{\calC}{\mathcal C}

\begin{document}

\title{One dimensional consensus based algorithm for non-convex optimization}

\author{Young-Pil Choi}
\address{Department of Mathematics, Yonsei University, Seoul 03722, Republic of Korea}
\email{ypchoi@yonsei.ac.kr}

\author{Dowan Koo}
\address{Department of Mathematics, Yonsei University, Seoul 03722, Republic of Korea}
\email{dowan.koo@yonsei.ac.kr}


\thanks{\textbf{Acknowledgment.} YPC has been supported by NRF grant (No. 2017R1C1B2012918), POSCO Science Fellowship of POSCO TJ Park Foundation, and Yonsei University Research Fund of 2020-22-0505. }

\begin{abstract}
We analyze the consensus based optimization method proposed in \cite{PTTM17} in one dimension. We rigorously provide a quantitative error estimate between the consensus point and global minimizer of a given objective function. Our analysis covers general objective functions; we do not require any structural assumption on the objective function.  
\end{abstract}


\maketitle


%
%
%
%
%
\section{Introduction}
 In the current work, we study the global minimization problem:
\[
\min_{x \in \R} f(x),
\]
where $f: \R \to \R_+$ is a given objective function that admits a unique global minimizer $x_* \in \R$. 

Global optimization problems appear in various research fields such as physics, economics, machine learning and artificial intelligence \cite{BDGG09,KE95,SLA12}. Among various optimization methods, we are interested in the consensus based optimization method  \cite{PTTM17}. More precisely, we consider an interacting particle system with position $x^i
_t \in \R$ at time $t$ governed by
\begin{equation}\label{main_eq}
\frac{dx_t^i}{dt} = -\lambda(x_t^i-m_t), \quad i =1,\dots,N, \quad t > 0,
\end{equation}
where $\lambda>0$ and $m_t$ is a weighted average of particles given by
\[
m_t = \sum_{i=1}^{N} x_t^i \psi_t^i, \qquad \psi_t^i= \frac{\omega_f^\alpha(x_t^i)}{\sum_{k=1}^{N}\omega_f^\alpha(x_t^k)}.
\]
Here the weight function $\omega_f^\alpha(x)$ is chosen as $\omega_f^\alpha(x)=\exp{(-\alpha f(x))}$ with $\alpha > 0$, based on the {\it Laplace principle} \cite{DZ10}: for any absolutely continuous probability measure $\rho$ with respect to the Lebesgue measure, it holds
\[
\lim_{\alpha \to \infty} \lt(-\frac1\alpha \log\lt(\int_\R \omega_f^\alpha \,d\rho \rt) \rt) = \inf_{x \in supp(\rho)} f(x).
\]
This shows that the weight function plays a crucial role in finding the global minimizer of the objective function for sufficiently large $\alpha >0$. Note that when $\alpha = 0$, $\psi^i_t \equiv 1/N$ for all $i=1,\dots,N$, and the uniform consensus point is just the arithmetic average of initial particles, which gives no information on the location of the global minimizer of a given objective function $f$, see Lemma \ref{lem_energy} (ii) below.

In \cite{PTTM17}, the particle system \eqref{main_eq} is first proposed and its mean-field equation is also derived. For the mean-field equation, the consensus behavior of solutions towards a neighborhood of the global minimizer is analyzed under certain assumptions on $f$ and $\alpha$. Later, the diffusive case is also considered in \cite{CCTT18, CJLZ21} at the continuum level. For the particle system \eqref{main_eq}, the global consensus behavior for the system \eqref{main_eq} is studied in a recent work \cite{HJK20}, however, it is not investigated the relation between the consensus point and the global minimizer of $f$. We refer to \cite{Tpre} and references therein for recent surveys on the consensus based optimization algorithm and its variants. 

The main purpose of this paper is to provide a sharp quantitative error estimate between the consensus point and the global minimizer for a large class of objective functions in one dimension. 

\begin{theorem}\label{main_thm} Let $a,b \in \R$ with $a<b$. Suppose $f \in \calC^2(a,b)\cap \calC[a,b]$ has a unique global minimizer $x_* \in (a,b)$ satisfying $f'(x_*)=0$ and $f''(x_*) \neq 0$. Consider the consensus based optimization model \eqref{main_eq} with $N=2$. If $x_* \in (x^1_0, x^2_0) \subset [a,b]$, then $x^1_t, x^2_t \to x_\infty$ as $t$ goes to infinity, and there exists $\alpha_1>0$ such that $|x_\infty - x_*| \leq \frac{c}{\sqrt{\alpha}}$ for any $\alpha>\alpha_1$ for some $c>0$ independent of $\alpha$.
\end{theorem}

\begin{remark} We can also consider the consensus based optimization model \eqref{main_eq} with $N>2$ to have a similar quantitative estimate for $|x_\infty - x_*|$. However, in this case, the bound may depend on some function $e=e(N)$ satisfying $e(N) \to \infty$ as $N \to \infty$. We give an example on this in Section \ref{ssec:N>2} below. This validates that taking into account the two-particle system is enough for the global optimization problems in one dimension, and moreover it gives an even better error estimate. 
\end{remark}

\begin{remark} We would like to emphasize that we do not require any structural assumption on the objective function. For the initial data, we only assume $x_* \in (x^1_0, x^2_0)$, i.e., the initial positions of particles do not need to be close to each other. 
\end{remark}

\begin{remark}
Our main strategy for the proof of Theorem \ref{main_thm} relies on the fact that the system \eqref{main_eq} is posed in one dimension. Note that if $x_\infty \neq x_*$, then, under the  assumptions of Theorem \ref{main_thm}, we have that at least one of two particles passes the point $x_*$ in a finite time. This observation plays a crucial role in analyzing the sharp quantitative error estimates between $x_\infty$ and $x_*$. Unfortunately, it is not clear how to employ this strategy for the multi-dimensional case $(d > 1)$.
\end{remark}

In the next section, we present {\it a priori} estimate showing the global consensus behavior for the system \eqref{main_eq} and the global existence and uniqueness of classical solutions. In Section \ref{sec:Pf}, we provide the details of the proof for our main theorem, Theorem \ref{main_thm}.

%
%
%
%
%

\section{Well-posedness and emergence of global consensus behavior}\label{sec:WP}

We first provide {\it a priori} estimates of solutions for the particle system \eqref{main_eq} which will be frequently used throughout this paper. We denote $\bfx_t := (x^1_t,\dots,x^N_t)$ in what follows.

\begin{lemma}\label{lem_energy} Let $\bfx_t$ be a classical solution of the system \eqref{main_eq}. Then the followings hold.
\begin{itemize}
\item[(i)] For any $i,j \in \{1,\dots,N\}$, we have
\[
x^i_t - x^j_t = (x^i_0 - x^j_0)\,e^{-\lambda t}, \qquad \forall \, t \geq 0.
\]
\item[(ii)] The average $\bar x_t:= \frac1N \sum_{j=1}^N x^j_t$ can be expressed as
\[
\bar x_t = \bar x_0 + \frac\lambda N \sum_{i,j=1}^N (x^i_0 - x^j_0) \int_0^t \psi^i_s e^{-\lambda s}\,ds, \qquad \forall \, t \geq 0.
\]
In particular, we have
\[
|\bar x_t| \leq | \bar x_0| + \lt(\max_{1 \leq i,j \leq N}|x^i_0 - x^j_0|\rt)(1 - e^{-\lambda t}), \qquad \forall \, t \geq 0.
\]
\end{itemize}
\end{lemma}
\begin{proof} (i) It is clear that for any $i,j \in \{1,\dots,N\}$
\[
(x_t^i - x^j_t)'  = - \lambda(x_t^i - x_t^j),
\]
where ${}' := \frac{d}{dt}$, 
thus solving this differential equation concludes the desired result.

(ii) We first find that $\bar x_t$ satisfies
\begin{equation}\label{est_barx}
\bar x_t ' = -\lambda (\bar x_t - m_t) = \lambda \sum_{i=1}^N (x^i_t - \bar x_t)\psi^i_t,
\end{equation}
due to $\sum_{i=1}^N \psi^i_t = 1$.
On the other hand, we obtain from (i) that 
\[
x^i_t - \bar x_t = \frac1N\sum_{j=1}^N (x^i_t - x^j_t) = \frac1N\sum_{j=1}^N (x^i_0 - x^j_0)\,e^{-\lambda t}.
\]
This combined with \eqref{est_barx} yields the first assertion in (ii). The second one simply follows due to  $\sum_{i=1}^N \psi^i_t = 1$.
\end{proof}
\begin{remark} (i) It is clear that the estimates in Lemma \ref{lem_energy} do not depend on dimensions, thus it also holds in higher dimensions. 

(ii) Lemma \ref{lem_energy} (i) only shows the consensus behavior of solutions, which does not provide any information on the location of the global minimizer of the given objective function $f$.
\end{remark}

As mentioned above, the particle system \eqref{main_eq} is taken into account in \cite{HJK20,PTTM17}, however, the well-posedness is not discussed. In this regard, we establish the well-posedness theory for the particle system \eqref{main_eq}.

\begin{theorem} Suppose that the objective function $f: \R \to \R_+$ is Lipschitz continuous. The particle system \eqref{main_eq} has a unique global-in-time classical solution $\bfx_t$ for any initial data $\bfx_0$ satisfying $|\bfx_0| < \infty$.
\end{theorem}
\begin{proof} Note that our particle system can be written in a matrix form as:
\[
\bfx_t' = -\lambda(I - M(\bfx_t))\bfx_t, 
\]
where $M=M(\bfx_t)$ is $N \times N$ real matrix whose entries are given by
$
M_{ij} = M(\bfx)_{ij} = \psi^j(\bfx)$, $1 \leq i,j  \leq N.
$
In particular, $M$ is a stochastic matrix, i.e. each entry of $M$ is non-negative and $M {\bf 1} = {\bf 1}$ where ${\bf 1}=(1,\dots,1)^T$ and is of a rank 1 since all rows are identical. 

We first prove the local-in-time existence and uniqueness of classical solutions. For this, it suffices to show that $F(\bfx) = -\lambda(I - M(\bfx))\bfx$ is Lipschitz continuous on each compact set $B(0,R):= \{\bfx \in \R^N: |x| \leq R\}$ for some $R>0$. Note that
$
        F(\bfx)-F(\bfy) = -\lambda(I - M(\bfx))(\bfx-\bfy) + \lambda(M(\bfx)-M(\bfy))\bfy$, $\bfx, \bfy \in B(0,R).
$
Since $\psi^i \leq 1$ for all $i=1,\dots, N$, the first term on the right hand side of the above can be bounded from above by $C|\bfx - \bfy|$ for some $C > 0$ independent of $R$. On the other hand, straightforward computations give
\[
M(\bfx)_{ij}-M(\bfy)_{ij} =\frac{\omega_f^\alpha(x^j) - \omega_f^\alpha(y^j)}{\sum_{k=1}^{N}\omega_f^\alpha(x^k)}(1-\psi^j(\bfy)) -\frac{\sum_{k\neq j} (\omega_f^\alpha(x^k)-\omega_f^\alpha(y^k))}{\sum_{k=1}^{N}\omega_f^\alpha(x^k)}\psi^j(\bfy).
\]
This, together with the fact that
$
|\omega^\alpha_f(x) - \omega^\alpha_f(y)| \leq \alpha |f(x) - f(y)| \leq \alpha Lip(f) |x-y|$, for all $x,y \in \R,
$
deduces
\[
|M(\bfx)_{ij}-M(\bfy)_{ij}| \leq \frac{\alpha Lip(f)}{\sum_{k=1}^{N}\omega_f^\alpha(x^k)}\sum_{k=1}^{N}|x^k - y^k| \leq \frac{\alpha Lip(f)}{\sqrt N e^{-\alpha \max_{B(0,R)}f}} |\bfx - \bfy|
\]
for all $i,j = 1,\dots,N$. Here $Lip(f)$ denotes the Lipschitz constant of $f$. Thus we get
$
|(M(\bfx)-M(\bfy))\bfy| \leq CR|\bfx - \bfy|
$
for some $C>0$. 

To extend this local-in-time solution to the global one, it is enough to show that there exists a $R>0$, independent of $t$, such that $|\bfx_t| \leq R$ for all $t \geq 0$. For any $i \in \{1,\dots,N\}$, we use Lemma \ref{lem_energy} to estimate
\begin{align*}
|x^i_t| &\leq |x^i_t - \bar x_t| + |\bar x_t| \cr
&\leq \frac1N \sum_{j=1}^N |x^i_0 - x^j_0| e^{-\lambda t} + | \bar x_0| + \max_{1 \leq i,j \leq N}|x^i_0 - x^j_0|(1 - e^{-\lambda t})  \cr
&\leq | \bar x_0| + \max_{1 \leq i,j \leq N}|x^i_0 - x^j_0|.
\end{align*}
Since the right hand side of the above inequality is independent of $t$, this concludes the desired result and completes the proof.
\end{proof}

%
%
%
%
%

\section{Proof of Theorem \ref{main_thm}}\label{sec:Pf}

We expect that the global consensus behavior of the system \eqref{main_eq} occurs (see Lemma \ref{lem_energy}), and its limit point is near the global minimizer $x_*$ with sufficiently large $\alpha > 0$. Since we are dealing with the one dimensional problem, if the global minimizer $x_*$ is  between initial positions of two particles, there are only two cases: either the consensus asymptotically occurs at the global minimizer $x_*$ or at least one of two particles hits the global minimizer in a finite time. Since there is nothing to prove in the first case, our analysis focuses on the moment when one particle just passes the global minimizer.

We first recast our two-particle system as
\begin{equation}\label{two_sys}
(x^1_t)'  = \lambda \psi^2_t(x^2_t-x^1_t) \quad \mbox{and} \quad (x^2_t)'= \lambda \psi^1_t(x^1_t-x^2_t).
\end{equation}

In the lemma below, we begin by dealing with some special objective functions. 

\begin{lemma}\label{lem1} Let $a,b \in \R$ with $a<b$, and suppose that the objective function $f$ is Lipschitz continuous on $[a,b]$. Moreover, we assume that there exist disjoint intervals $A,B \subset [a,b]$, and a positive constant $c_f$ such that $|f(x)-f(y)|\geq c_f|x-y|$ for all $x \in A$ and $y \in B$. If there exist  nonnegative constants $0 \leq T_0 < T \leq \infty$ such that $(x^1_t, x^2_t) \in A \times B$ for all $t \in [T_0,T)$, then we have $|x^1_T - x^1_{T_0}| \leq \frac{\ln 2}{\alpha c_f}$ if $f(A) < f(B)$\footnote{For sets $E, F \subset \R$, we denote by $E < F$ if $x<y$ for all $x \in E$ and $y \in F$, and $f(E) := \{f(x):x\in E\}$.} and $|x^2_T - x^2_{T_0}| \leq \frac{\ln 2}{\alpha c_f}$ if $f(B) < f(A)$. In particular, if $T_0 = 0$, $T = \infty$, and the initial data are given by  $x^1_0=x_* = a$, then we have
\begin{equation}\label{bdd_lin}
|x_\infty-x_*|\leq \frac{\ln 2}{\alpha c_f}.
\end{equation}
\end{lemma}
\begin{proof} Under the assumption on $f$, we first readily find $f(A) < f(B)$ or $f(B) < f(A)$. Without loss of generality, we may assume that $A<B$ and $f(A) < f(B)$. Indeed, a simple modification of the following argument can be applied to the other three cases.

By Lemma \ref{lem_energy} (i), we first find that $a \leq x^1_{T_0} \leq x^1_{t} \leq x^2_{t} \leq x^2_{T_0} \leq b$ for all $t \in [T_0,T)$. Then it follows from \eqref{two_sys} that 
\begin{equation}\label{gd_eq}
0 \leq x^1_{T} -x^1_{T_0} = \int_{T_0}^T  \frac{dx^1_t}{dt}\,dt = \lambda\int_{T_0}^T \psi^2_t (x^2_t - x^1_t)\,dt.
\end{equation}
We observe that for $t \in [T_0,T)$,
\[
\psi^2_t=\frac{\omega^\alpha_f(x^2_t)}{\omega^\alpha_f(x^1_t)+\omega^\alpha_f(x^2_t)}=\frac{1}{\omega^\alpha_f(x^1_t)/\omega^\alpha_f(x^2_t)+1}
\leq \frac{1}{e^{\alpha(f(x^2_t)-f(x^1_t))} +1} 
\leq \frac{1}{e^{\alpha c_f(x^2_t-x^1_t)} +1}, 
\]
and
$x^2_t-x^1_t=(x^2_{T_0}-x^1_{T_0})e^{-\lambda (t-T_0)}$. Thus, by the change of variables  $\tau:= (x^2_{T_0}-x^1_{T_0})e^{-\lambda (t-T_0)}$, we obtain
\begin{align*}
x^1_{T} -x^1_{T_0} &\leq \lambda\int_{T_0}^{T} \frac{x^2_t - x^1_t}{e^{\alpha c_f(x^2_t-x^1_t)} +1}\,dt \leq \int_{0}^{x^2_{T_0}-x^1_{T_0}} \frac{1}{e^{\alpha c_f\tau}+1}\,d\tau \cr
&=x^2_{T_0}-x^1_{T_0} - \frac{1}{\alpha c_f} \lt(\ln(e^{\alpha c_f(x^2_{T_0}-x^1_{T_0})} +1) - \ln 2\rt)  \leq \frac{\ln 2}{\alpha c_f}.
\end{align*}
We notice that the upper bound is independent of both $T_0$ and $T$. If $T_0 = 0$, $T = \infty$, and the initial data are given by  $x^1_0=x_* = a$, then we deduce the estimate \eqref{bdd_lin}. This completes the proof.
\end{proof}

\begin{remark} (i) In Lemma \ref{lem1}, we do not require any largeness of $\alpha$ for the error estimate. 

(ii) Let us comment on the upper bound estimate in Lemma \ref{lem1}. If the objective function $f$ is simply given by a straight line, i.e. $f(x)=x$ and $x^1_0 = x_*< x^2_0$, then we can find 
\[
|x_\infty - x_*| \geq \frac{\ln2-\e}{\alpha},
\] 
where $\e>0$ can be arbitrarily small provided that $\alpha (b-a)$ is large enough. Indeed, we can use almost the same argument as in Lemma \ref{lem1} to get
\[
 x_\infty - x_* =   \int_0^{b-a} \frac{d\tau}{1+e^{\alpha  \tau}} =b-a + \frac{1}{\alpha  }\lt(\ln 2 - \ln\lt(1 + e^{\alpha   (b-a)}\rt)\rt) =  \frac{1}{\alpha}\lt(\ln 2 - \ln\lt( 1+ \frac{1}{e^{\alpha  (b-a)}}\rt)\rt).
\]
Since we are considering the case $\alpha > 0$ large enough, this shows that the result of Lemma \ref{lem1} is quite sharp and $|x_\infty - x_*| = \Theta(\alpha^{-1})$ for $\alpha \gg 1$.
\end{remark}

We next take into account the convex optimization problem, i.e. the objective function $f$ is assumed to be convex. 

\begin{lemma}\label{lem2} Let $a,b \in \R$ with $a<b$, and suppose $f \in \calC^2(a,b)\cap \calC[a,b]$. Moreover, we assume that there exists a positive constant $c_f$ such that $f$ satisfies $f''(x)\geq c_f>0$ in $(a,b)$ and $f'(x_*)=0$ for $x_*\in (a, b)$. If $x_* \in (x^1_0, x^2_0) \subset [a,b]$, then we have 
\[
|x_\infty-x_*|\leq \frac{\sqrt{\ln 2}}{\sqrt{\alpha c_f}}.
\]
\end{lemma}
\begin{proof} If $x_\infty = x_*$, i.e. two particles converges to $x_*$ as $t \to \infty$, we are done. Otherwise, exactly one of two particles hits the global minimizer $x_*$ in a finite time. Without loss of generality, we may assume that $x^1_t$ hits $x_*$ at some time $T_0$, i.e. $x^1_{T_0} = x_*$ and $x^1_t > x_*$ for $t > T_0$. Let $A = [x_*, x_\infty)$ and $B = (x_\infty, x^2_{T_0}]$. Then, for $t > T_0$, we obtain $x^1_t \in A$ and $x^2_t \in B$. On the other hand, since $f$ is increasing on $[x_*, b]$, $\psi^2_t \leq \psi^1_t$, and this implies
\[
x_\infty - x^1_t = \lambda\int_t^\infty   \psi^2_s(x^2_s-x^1_s)\,ds \leq \lambda\int_t^\infty   \psi^1_s(x^2_s-x^1_s)\,ds = x^2_t - x_\infty.
\]
Thus we have
\[
f(x^2_t) - f(x^1_t) \geq f'(x^1_t)(x^2_t - x^1_t) + \frac12 c_f(x^2_t - x^1_t)^2 \geq c_f (x^1_t - x_*)(x^2_t - x^1_t) + \frac12 c_f(x^2_t - x^1_t)^2 \geq c_f(x_\infty - x_*)(x^2_t - x^1_t) 
\]
for $t > T_0$, due to $f'(x_*) = 0$. Then, by Lemma \ref{lem1}, we obtain
\[
x_\infty - x_* \leq \frac{\ln 2}{\alpha c_f(x_\infty - x_*)}.
\]
This completes the proof.
\end{proof}
\begin{remark}\label{rmk_cov} Similarly as before, we discuss the sharpness of the error estimate obtained in Lemma \ref{lem2}. Let us consider the quadratic objective function $f = x^2$  and assume $a = x_*=0 < b$. Then we claim that 
\[
|x_\infty - x_*| = x_\infty \geq \frac{c_0 - \epsilon}{\sqrt{\alpha}},
\]
where $c_0 > 0$ is independent of $\alpha$, and $\epsilon > 0$ can be arbitrarily small provided that $\sqrt{\alpha} b$ is large enough.

We notice that $0 \leq x^1_t \leq x_\infty$ and
\[
\psi^2_t = \frac{1}{1+e^{\alpha((x^2_t)^2-(x^1_t)^2)}} = \frac{1}{1+e^{\alpha  (\tau + 2x^1_t)\tau}},
\]
where $\tau = x^2_t-x^1_t > 0$. This together with the identity \eqref{gd_eq} yields
\[
x_\infty = x_\infty -x_* = \int_0^b \frac{d\tau}{1+e^{\alpha  (\tau + 2x^1_t)\tau}} \geq \int_0^\infty \frac{d\tau}{1+e^{\alpha  (\tau + 2x_\infty)\tau}} - \int_b^\infty \frac{d\tau}{1 + e^{\alpha \tau^2}}.
\]
Here we estimate 
\[
\int_0^\infty \frac{d\tau}{1+e^{\alpha  (\tau + 2x_\infty)\tau}} \geq \frac12 \int_{2x_\infty}^\infty e^{-\alpha \tau^2} d\tau = \frac12\int_0^\infty e^{-\alpha \tau^2}\,d\tau - \frac12 \int_0^{2x_\infty} e^{-\alpha \tau^2}\,d\tau \geq \frac{\sqrt\pi}{4\sqrt\alpha} - x_\infty
\]
and
\[
\int_b^\infty \frac{d\tau}{1 + e^{\alpha \tau^2}} \leq \int_b^\infty e^{-\alpha \tau^2}d\tau \leq \int_b^\infty \frac{\tau}{b}e^{-\alpha \tau^2}d\tau = \frac{1}{2\alpha b} e^{-\alpha b^2}.
\]
Hence we have
\[
|x_\infty - x_*|\geq \frac1{4\sqrt\alpha}\lt( \frac{\sqrt\pi}{2} - \frac{1}{\sqrt{\alpha} b} e^{-\alpha b^2}\rt),  
\]
and this concludes the desired result.
\end{remark}

We now want to apply the previous estimates to treat the general case. Before proceeding, we introduce a definition of a set called {\it calyx}.

\begin{definition}  If $f \in \calC^2(a,b) \cap \calC[a,b]$ has a unique global minimizer $x_*$ satisfying $f'(x_*)=0$, we say $[x_*-r, x_*+r]$ is a {\it calyx} of $f$ if there exist $c,r>0$ such that $f$ is convex on $[x_*-r,x_*+r]$ and $f(x)-f(y) \geq c|x-y|$ for all $x,y \in [a,b]$ satisfying either $x-x_*\geq r \geq y-x_*\geq 0$ or $x-x_* \leq -r \leq y-x_* \leq 0$.
\end{definition}

We then present a characterization for a $\calC^2$-function with a unique global minimizer. 

\begin{lemma}\label{lem_cal} Let $a,b \in \R$ with $a<b$, and suppose $f \in \calC^2(a,b)\cap \calC[a,b]$ has a unique global minimizer $x_*$. We further assume that $x_*$ is a local minimizer as well, i.e. $f'(x_*)=0$, and $x_*$ might be the end point $a$ or $b$. If $f''(x_*) \neq 0$, then there exist $c>0$ and $r_1>r_2>0$ such that $f''(x) \geq c$ for $|x-x_*| \leq r_1$ and the interval $[x_* - r_2, x_* + r_2]$ is the calyx of $f$.
\end{lemma}
\begin{proof}
Since $f \in \calC^2(a,b)$ and $f''(x_*) \neq 0$, there exist positive constants $C_1,c_1$, and $r_1$ such that $c_1 \leq f''(x) \leq C_1$ for all $|x-x_*| \leq r_1$. This, together with Taylor's theorem and $f'(x_*)=0$, yields
\begin{equation}\label{est_1}
  f_* + \frac{1}{2}c_1(x-x_*)^2 \leq f(x) \leq f_* + \frac{1}{2}C_1(x-x_*)^2 
\end{equation}
for all $|x-x_*| \leq r_1$, where $f_* := f(x_*)$. On the other hand, by the assumption that $f$ attains its global minimum at a unique point $x_*$, we obtain
\begin{equation}\label{est_2}
    f_1:= \inf_{[a,b]\setminus (x_*-r_1,x_*+r_1)}f > f_*. 
\end{equation}
We then choose $0<r_2 = \min\{ \sqrt{\delta/C_1}, r_1\}$  with $\delta : = f_1 - f_* >0$ to deduce $f(x_* + r_2) \leq f_* + \delta/2 < f_1 \leq f(x_* + r_1)$. In particular, we have $r_1>r_2>0$.

Let us now assume $x-x_*\geq r_2 \geq y-x_*\geq 0$. We then consider the following two cases.

(i) ($x-x_* \geq r_1$): In this case, it follows from \eqref{est_1} and \eqref{est_2} that
\[
   f(x)-f(y) \geq f_1 - (f_*+\frac{1}{2}\delta)=\frac{1}{2}\delta \geq  \frac{\delta}{2(b-x_*)}(x-y),
\]
due to $x-y \leq b-x_*$.

(ii) ($r_2 \leq x-x_* < r_1$): By Taylor's theorem and the assumption $f'(x_*) = 0$, we estimate
\[
f(x)-f(y) \geq f'(y)(x-y)+\frac{1}{2}c_1(x-y)^2 \geq c_1(y-x_*)(x-y)+\frac{1}{2}c_1(x-y)^2 \geq \frac{1}{2}c_1(x-x_*)(x-y) \geq  \frac{1}{2}c_1r_2(x-y).
\]
Applying a similar argument to the case $x-x_* \leq -r_2 \leq y-x_* \leq 0$ and combining the resulting estimate with the above yield $f(x)-f(y) \geq c_2|x-y|$ with 
\[
c_2 = \min \left\{\frac{\delta}{2\max\{(b-x_*),(x_*-a)\}}, \ \frac{1}{2}c_1r_2\right\} > 0.
\]
\end{proof}

We are now ready to provide the details of proof for our main result.

\begin{proof}[Proof of Theorem \ref{main_thm}] 
If $x_\infty = x_*$, then we are done. Otherwise, one of two particles passes the point $x_*$ in a finite time. By the initialization procedure, without loss of generality, we may assume that $x^1_0 =x_*$. 

We invoke Lemma \ref{lem_cal} and consider two cases: (i) $x^2_0 - x_* \leq r_2$ and (ii) $x^2_0 - x_* > r_2$. Since the case (i) implies that $x_\infty$ lies in a calyx of $f$, thus the desired error estimate follows by employing Lemma \ref{lem2}. On the other hand, for the second one, we take $\alpha_0 > 0$ such that $\frac{1}{\alpha_0 c_2}=r_2$. We then claim for any $\alpha > \alpha_0$, $x^2_t$ passes the point $x_*+r_2$ in a finite time. Suppose not, i.e. either $x^1_t$ passes $x_*+r_2$ in a finite time or $x^1_t$ converges to that point as time goes to infinity. This implies that 
there exists a $T_1 \in (0,\infty]$ such that $x^1_{T_1} = x_* + r_2$, then it is clear $x^1_t-x_* \leq r_2 \leq x^2_t-x_*$ for $t < T_1$. Thus, by Lemmas \ref{lem1} and \ref{lem_cal}, we have
\[
r_2 = x^1_{T_1}-x_* \leq \frac{\ln 2}{\alpha c_2} <  \frac{\ln2 }{\alpha_0 c_2} < r_2,
\]
and this leads to a contradiction. Thus we only consider the case that $x^2_t$ passes $x_*+r_2$ in a finite time, say $T_2$, for any $\alpha > \alpha_0$. Note that, similarly as the above, we get 
\[
x^1_{T_2} - x_* \leq \frac{\ln 2}{\alpha c_2}
\] 
for $\alpha > \alpha_0$. On the other hand, for $t \geq T_2$, we can apply Lemma \ref{lem2} to deduce 
\[
x^1_t - x^1_{T_2} \leq \frac{\sqrt{\ln 2}}{\sqrt{\alpha c_1}}.
\]
We finally combine the above two estimates and let $t \to \infty$ to conclude our desired result.
\end{proof}

%
%
%
%
%

\subsection{Remark on the case $N>2$}\label{ssec:N>2} In this subsection, we discuss the analysis of the consensus based optimization model \eqref{main_eq} with $N > 2$. It is observed in \cite{CCTT18, PTTM17} that considering many interacting particles may be helpful to find the global minimizer since they can explore a larger portion of the landscape of the graph of the objective function $f$. 

However, in the one dimensional case, we find that the two-particle consensus based optimization model provides a better information on the global minimizer. As stated in the example below, in some special case, the error between $x_\infty$ and $x_*$ increases as the number of particles $N$ gets larger.

\begin{example} Let $a<b$, $f(x) = x$, and fix $j \in \{1,\dots,N\}$. Suppose that the initial data $\bfx_0$ is given by $x^i_0 = a$ for $i=1,\dots,j$ and $x^i_0 = b$ otherwise. In this case, by Lemma \ref{lem_energy} (i), we readily get $x^i_t = x^1_t$ for $i=1,\dots,j$ and $x^i_t = x^N_t$ otherwise for all $t \geq 0$. Note that $a = \argmin_{[a,b]}f$, and 
\[
\psi^i_t = \frac{\omega^\alpha_f(x^i_t)}{j\omega^\alpha_f(x^1_t) + (N-j)\omega^\alpha_f(x^N_t)}, \qquad i=1,\dots,N, \quad t \geq 0.
\]
Similarly as in Lemma \ref{lem1}, we get
\[
0 \leq x_\infty - x_* = \lambda\sum_{i=1}^N\int_{0}^{\infty} \psi^i_t (x^i_t-x^1_t)\,dt = \lambda(N-j)\int_{0}^{\infty} \psi^N_t (x^N_t-x^1_t)\,dt,
\]
where 
\[
\psi^N_t = \frac{1}{j e^{\alpha(x^N_t - x^1_t)} + (N-j)}.
\]
Then we again use the change of variable $\tau := x^N_t-x^1_t = (b-a)e^{-\lambda t}$ to estimate
\[
x_\infty - x_* = (N-j) \int_0^{b-a} \frac{d\tau}{j e^{\alpha \tau} + (N-j)} = \frac1\alpha \ln \lt(\frac{N}{j + (N-j)e^{-\alpha (b-a)}} \rt).
\]
Thus, if $j=1$ and $\alpha \gg 1$, then the error $|x_\infty - x_*|$ has an asymptotic order $O((\ln N)/\alpha)$. 
\end{example}
\begin{remark} One can also use the similar argument as in the above example combined with that of Remark \ref{rmk_cov} to consider the have an asymptotic order $O(\sqrt{(\ln N)/\alpha})$ for the error $|x_\infty - x_*|$ when $f = x^2$.
\end{remark}


\end{document}